\documentclass{amsart}

\usepackage{graphicx,graphics, amsthm, amsfonts, amsbsy,amssymb, amsmath, cite, array}
\everymath{\displaystyle}
\usepackage[margin=1in]{geometry}
\usepackage{tikz}
\usetikzlibrary{calc}

\allowdisplaybreaks
\renewenvironment{proof}{{\noindent\bfseries Proof.}}{\qed}

\newtheorem{theorem}{Theorem}[section]

\newtheorem{lemma}[theorem]{Lemma}
\newtheorem{corollary}[theorem]{Corollary}

\def\qed{\nolinebreak\hfill\rule{.2cm}{.2cm}\par\addvspace{.5cm}}

\title{Betti numbers of skeletons of thick trees}

\author {Ralf Fr\"oberg}
\address{Department of Mathematics, Stockholm University, Stockholm, Sweden}
\email{frobergralf@gmail.com}
\date{}
\begin{document}
\maketitle

\begin{abstract}
A tree could be defined as follows. An edge is a tree. If $T_{k-1}=\cup_{i=1}^{k-1}e_i$ is a tree with $k-1$ edges $e_i$, and $e_k$ an edge, then $T_k=T_{k-1}\cup e_k$ is a tree if
$T_{k-1}\cap e_k$ is a point.
We generalize this construction: A simplex $S_1$ of dimension $\ge1$ is a thick tree. If $G_{k-1}=\cup_{i=1}^{k-1}S_i$ is a thick tree, where $S_i$ are
simplices of dimension $\ge1$, and $S_k$ a new simplex of dimension $\ge1$,
then $G_{k-1}\cup S_k$ is a thick tree if $G_{k-1}\cap S_k$ is a point.
All homological properties of Stanley-Reisner rings of thick trees are well known. 
We determine the Hilbert series and Betti numbers for Stanley-Reisner rings of skeletons of thick trees. From this one can read of projective dimension, regularity,
and judge when they are Cohen-Macaulay.
\end{abstract}

\section{Preliminaries}
An abstract simplicial complex $\Sigma$ on a vertex set $V=\{ x_1,\ldots,x_N\}$ is a set of subsets of $V$ such that $\{x_i\}\in\Sigma$ for all $i$, and if $\sigma\in\Sigma$
and $\tau\subset\sigma$ then $\tau\in\Sigma$. The subsets $\sigma$ are called faces, and the faces that are maximal with respect to inclusion are called
facets. A usual way to give a simplicial complex is to give its facets. Thus the simplicial complex $\{x_1,x_2,x_3\},\{x_1,x_4,x_5\}, \{x_5,x_{6}\}$ has the faces $\{x_1,x_2,x_3\},
\{x_1,x_4,x_5\},\{x_1,x_2\},\{x_1,x_3\},\{x_2,x_3\},\{x_1,x_4\},\{x_1,x_5\},\{x_4,x_5\},\{x_5,x_6\}, 
\{x_1\},\{x_2\},\{x_3\},\\
\{x_4\},\{x_5\},\{x_6\},\emptyset$. A geometric representation of this complex, which is a thick tree, is two triangles with one common point $\{x_1\}$, and an edge
$\{x_5,x_6\}$ (see Figure \ref{complex}). It is a thick tree, i.e. of the form $\cup_{i=1}^kS_i$, where $S_i$ are simplices of dimension $\ge1$ and such that
$\cup_{i=1}^l S_i\cap S_{l+1}$ is a point for $l=1,\ldots,k$.
The complex in the figure is built in the following way. We start with the simplex $\{x_1,x_2,x_3\}$. Then attach the
simplex $\{x_1,x_4,x_5\}$ in  $\{x_1\}$. Finally attach the edge $\{x_5,x_6\}$ in $\{x_5\}$.

\begin{figure}[h]
	\centering
	\begin{tikzpicture}[scale=1.1, line join=round, line cap=round]
		\coordinate (x1) at (0,0);
		
		\coordinate (x3) at (-2,1.2);
		\coordinate (x2) at (-2,-1.2);
		
		\coordinate (x4) at (2,1.2);
		\coordinate (x5) at (2,-1.2);
		
		\coordinate (x6) at (0,-2.2);
		
		\fill[blue!15] (x1) -- (x2) -- (x3) -- cycle;
		\fill[blue!15] (x1) -- (x4) -- (x5) -- cycle;
		
		\draw[thick] (x1) -- (x2) -- (x3) -- cycle;
		\draw[thick] (x1) -- (x4) -- (x5) -- cycle;
		\draw[thick] (x5) -- (x6);
		\foreach \p in {x1,x2,x3,x4,x5,x6} {
			\fill (\p) circle (2pt);
		}
		
		\node[font=\small, above] at (x1) {$x_1$};
		\node[font=\small, below] at (x2) {$x_2$};
		\node[font=\small, above] at (x3) {$x_3$};
		\node[font=\small, above] at (x4) {$x_4$};
		\node[font=\small, below] at (x5) {$x_5$};
		\node[font=\small, above] at (x6) {$x_6$};
		
	\end{tikzpicture}
	\caption{Pictorial representation of the simplicial complex $\{x_1,x_2,x_3\},\{x_1,x_4,x_5\}, \{x_5,x_{6}\}$.}
\label{complex}
\end{figure}
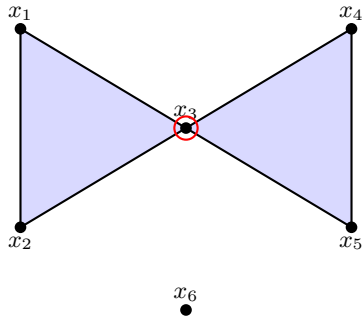

The dimension of a face is the number of its vertices minus one, and the dimension of a simplicial complex is the maximal dimension of its facets.

The Stanley-Reisner ring  $K[\Sigma]$ over a field $K$ of a simplicial complex
$\Sigma$ with vertex set $\{x_1,\ldots,x_N\}$ is $K[x_1,\ldots,x_N]/I$, where $I$ is generated by all squarefree monomials $x_{i_1}\cdots x_{i_n}$ for which
$\{x_{i_1},\ldots,x_{i_n}\}$ is not a face of $\Sigma$. For the example above, $K[x_1,\ldots,x_6]/(x_6(x_1,x_2,x_3,x_4)+(x_2,x_3)(x_4,x_5))$ is the Stanley-Reisner ring.
The Stanley-Reisner ring is often used to compare topological properties of the geometric representation of the
complex with algebraic properties of its Stanley-Reisner ring. For a graded $K$-algebra $R=\oplus_{i\ge0}R_i$, the Hilbert series of $R$ is the formal power series
$R(t)=\sum_{i\ge0}\dim_K(R_i)t^i$.

If $f_i=f_i(\Sigma)$ is the number of $i$-dimensional faces in $\Sigma$, then the $f$-vector of $\Sigma$ is $(f_{-1},f_0,\ldots,f_{\dim\Sigma})$, 
where $f_{-1}=1$, counting the empty set.

The following lemma is well known and easy to prove.

\begin{lemma}\label{fvector}
If $\Sigma$ is a simplicial complex with $f$-vector $(1,f_0,f_1,\ldots,f_d)$, then the Hilbert series of $K[\Sigma]$ is
$$1+\frac{f_0t}{1-t}+\frac{f_1t^2}{(1-t)^2}+\cdots+\frac{f_dt^{d+1}}{(1-t)^{d+1}}.$$
\end{lemma}

\medskip
If $\sigma$ is a face in $\Sigma$, then the link of $\sigma\in\Sigma$ is $\{\tau\in\Sigma;\tau\cap\sigma=\emptyset,\tau\cup\sigma\in\Sigma\}$.

\medskip
For a simplicial complex $\Sigma$, the $k$-skeleton, which we denote by $\Sigma^{(k)}$, consists of all faces in $\Sigma$ of dimension $\le k$.
If $\Sigma$ has $f$-vector $(1,f_0,f_1,\ldots,f_d)$, then $\Sigma^{(k)}$ has $f$-vector $(1,f_0,\ldots,f_k)$ if $k\le\dim\Sigma$.

A simplicial complex of dimension one is a graph. There is another common way to compare properties of a graph $(V,E)$ with vertex set $V=\{x_1,\ldots,x_N\}$
and edge set $E$ with algebraic properties of an ideal in an algebra, the edge ideal. The edge ideal is generated by all $x_ix_j$ in $K[x_1,\ldots,x_N]$ for which 
$\{x_i,x_j\}\in E$. We will all the time consider the Stanley-Reisner ideal if the dimension of the complex is one, not the edge ideal.

If $I$ is a graded ideal in $S=K[x_1,\ldots,x_N]$, then $S/I$, as any finitely generated graded $S$-module, has a minimal resolution

$$0\leftarrow S/I\leftarrow S\leftarrow \bigoplus_{j\ge2}S[-j]^{\beta_{1,j}}\leftarrow\cdots\leftarrow\bigoplus_{j\ge p+1}S[-j]^{\beta_{p,j}}\leftarrow 0$$
where $S[-j]$ means that the degrees of $S$ are shifted so that $S[-j]_d=S_{d-j}$. The numbers $\beta_{i,j}$ are called the graded Betti numbers of $S/I$.
Using that for each degree of the resolution, we have an exact sequence of vector spaces,
so the alternating sum of their dimensions is 0, we get the well known formula for the Hilbert series of $S/I$ 

\begin{equation}\label{hilbser}
S/I(t)=\sum_{i,j}(-1)^i\beta_{i,j}t^j/(1-t)^N.
\end{equation}

The projective dimension pd$(S/I)$ of $S/I$ is $\max\{i;\beta_{i,j}\ne0\}$, the regularity of $S/I$ is $\max\{j-i;\beta_{i,j}\ne0\}$. 
If $S/I$ has a minimal resolution as above, then $I$ has a minimal resolution
$$0\leftarrow I\leftarrow \bigoplus_{j\ge2}S[-j]^{\beta_{1,j}}\leftarrow\cdots\leftarrow\bigoplus_{j\ge p+1}S[-j]^{\beta_{p,j}}\leftarrow 0,$$
so pd$(I)=$pd$(S/I)-1$ and reg$(I)$=reg$(S/I)+1$.
If the minimal resolution of $I$ looks like this:
$$0\leftarrow I\leftarrow S^{b_1}[-t]\leftarrow S^{b_2}[-t-1]\leftarrow\cdots\leftarrow S^{b_l}[-t-l+1]\leftarrow 0,$$
then $I$ (and $S/I$) is said to have a $t$-linear resolution. It follows that in this case the only Betti numbers $\beta_{i,j}(S/I)$, $i>0$, are
$\beta_{i,i+t-1}$, and knowing the Betti numbers is equivalent to knowing the Hilbert series by equation \ref{hilbser}.

\medskip
Of fundamental importance is the following result of Hochster, \cite{ho}.

\begin{theorem}
For a simplicial complex $\Sigma$ on $V$
$$\beta_{i,j}(K[\Sigma])=\sum_{\overset{S\subseteq V}{|S|=j}}\dim_K\widetilde H_{j-i-1}(\Sigma_S;K),$$
where $\Sigma_S$ is the induced subcomplex on $S$, and $\widetilde H$ stands for reduced homology.
\end{theorem}

Thick trees are a special case of simplicial complexes called fat forests. Fat forests are defined recursively as follows. 
A $d$-simplex (i.e. with $d+1$ vertices) $F_1$ of dimension $\ge0$ is a fat forest.
If $F_i$, $i=1,\ldots,k$, are faces and $G_{k-1}=F_1\cup\cdots\cup F_{k-1}$ is a fat forest, then $G_{k-1}\cup F_k$ is a fat forest if $H=G_{k-1}\cap F_k$ is a simplex, $\dim H\ge -1$. 
(If $\dim H=-1$, then $G_{k-1}$ and $F_k$ are disjoint.) It is known \cite{fr,fr1} that fat forests are precisely the simplicial complexes 
for which the Stanley-Reisner rings have 2-linear resolutions. If $S/I$ is a
Stanley-Reisner ring with 2-linear resolution, then the ideal has generators of degree 2, and can thus be interpreted as an edge ideal of a graph. These graphs $G$
are exactly \cite{fr,fr1} the graphs for which the complement graph is chordal. (The complement graph $\overline G$ to $G$ has the same vertices as $G$, and $\{x_i,x_j\}$ is
an edge in $\overline G$ if and only if $\{x_i,x_j\}$ is not an edge in $G$.) A graph is chordal if any cycle with at least four vertices has a chord, or, equivalently,
if the only induced cycles have length three. In \cite{fr1} the Hilbert
series, the Betti numbers, the projective dimension, and the depth of Stanley-Reisner rings $S/I$ of fat forests are determined. The regularity of $S/I$ is one, so reg$(I)=2$.

\section{Prerequisites}
Our starting point is thick trees. Thick trees are those fat forests for which $G_{k-1}\cap F_k$ is a point for each $k$. We now introduce a notation for thick trees.
Let $\Delta(n_1,\ldots,n_e)$ be a simplicial complex with $e$ facets $S_i$, $\dim S_i=n_i-1\ge1$ (so with $n_i\ge2$ vertices), such that $(S_1\cup\cdots\cup S_i)\cap S_{i+1}$ 
is a point for all $i$, $i=1,\ldots,e-1$. When all $n_i=n$, we denote the complex by $\Delta(n,e)$. We have that 
$\Delta(n_1,\ldots,n_e)$ has $N=\sum_{i=1}^en_i-(e-1)$ vertices and $\dim\Delta(n_1,\ldots,n_e)=\max\{n_i\}-1$. The following Lemma is a special case of \cite[Theorem 1]{fr1}.

\begin{lemma}\label{spec} The Hilbert series of $K[\Delta(n_1,\ldots,n_e)]$ is $$\sum_{i=1}^e\frac{1}{(1-t)^{n_i}}-\frac{e-1}{1-t}.$$ 
pd$(K[\Delta(n_1,\ldots,n_e)])=N-2,$ depth($K[\Delta(n_1,\ldots,n_e)])=2$, and reg$(K[\Delta(n_1,\ldots,n_e)])=1$.
\end{lemma}
We denote the $k$-skeleton of $\Delta(n_1,\ldots,n_e)$ by $\Delta(n_1,\ldots,n_e)^{(k)}$.
Our aim is to study the algebraic properties of the Stanley-Reisner rings $K[\Delta(n_1,\ldots,n_e)^{(k)}]$. 

\section{Results}
The following lemma is trivial.

\begin{lemma}\label{size}
$\Delta(n_1,\ldots,n_e)^{(k)}$ has $f$-vector $(1,\sum_{i=1}^en_i-(e-1),\sum_{i=1}^e{n_i\choose2},\ldots,\sum_{i=1}^e{n_i\choose s+1})$, where $s=\min\{k,\max\{n_i\}-1\}$.
\end{lemma}

The lemmas \ref{fvector} and \ref{size} give the following theorem.

\begin{corollary}\label{hill}
$K[\Delta(n_1,\ldots,n_e)^{(k)}]$ has Hilbert series 
$$1+\frac{Nt}{1-t}+\frac{\sum_{i=2}^{s+1}c_it^i}{(1-t)^i}=\frac{(1-t)^N+Nt(1-t)^{N-1}+\sum_{i=2}^{s+1}c_it^i(1-t)^{N-i}}{(1-t)^N},$$
where $s=\min\{ k,\max\{n_i\}-1\}$, $N=\sum_{i=1}^en_i-(e-1)$, and $c_j=\sum_{i=1}^e{n_i\choose j}$.
\end{corollary}

Some of these skeletons are studied earlier. If $k\ge \max\{n_i\}-1$, then the $k$-skeleton of $\Delta(n_1,\ldots,n_e)$ is equal to
$\Delta(n_1,\ldots,n_e)$, which is a thick tree. If $k=0$, the $k$-skeleton of $\Delta(n_1,\ldots,n_e)$ is just a set of $N$
vertices, which is another example of a fat forest. If $k=1$, then, if $\Delta(n_1,\ldots,n_e)=\cup_{i=1}^e S_i$, and the points $(S_1\cup\cdots\cup S_i)\cap S_{i+1}$ are all different,
the $k$-skeleton $\Delta(n_1,\ldots,n_e)^{(1)}$ is studied in \cite{ra} under the name of Wollastonite graphs.

Let $e\ge 2$ and let $n_1,n_2,\dots,n_e\ge 2$ be integers.  The \emph{Wollastonite graph} $W(n_1,n_2,\dots,n_e)$ \cite{ra} is obtained by taking
pairwise vertex-disjoint complete graphs $K_{n_1},\,K_{n_2},\,\dots,\,K_{n_e},$ 
and for each $i=1,2,\dots,e-1$ identifying (gluing) exactly one vertex $v_i$ of
$K_{n_i}$ with one vertex of $K_{n_{i+1}}$.
Moreover, the chosen vertices $v_i$ in the intermediate blocks are all distinct.
In \cite{ra} the algebraic properties of their edge ideals, i.e., the ideals generated by $\{ x_ix_j\}$ for all edges  
$\{x_i,x_j\}$ in $\Delta(n_1,\ldots,n_e)_{(1)}$, while we study the Stanley-Reisner ideals. The edge ideal
is generated by all $x_ix_j$ where $x_i$ and $x_j$ belong to the same $K_{n_i}$, while the Stanley-Reisner ideal is generated by those $x_ix_j$
for which $x_i$ and $x_j$ do not belong to the same $K_{n_i}$.
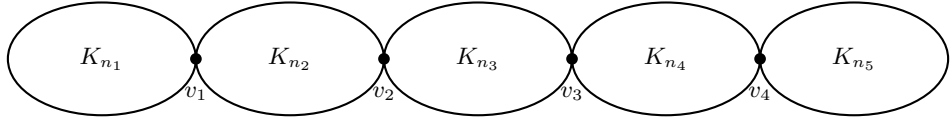
\begin{figure}[h]
	\centering
	\begin{tikzpicture}[scale=1, line cap=round, line join=round, font=\small]
		
		\def\R{1.25}   
		\def\r{0.75}   
		\def\dx{2.6}   
		\def\d{5}      
		
		\coordinate (C1) at (0,0);
		\coordinate (C2) at (\dx-.1,0);
		\coordinate (C3) at (2*\dx-.2,0);
		\coordinate (C4) at (3*\dx-.3,0);
		\coordinate (C5) at (4*\dx-.4,0);
		
		\foreach \i/\Ci in {1/C1,2/C2,3/C3,4/C4,5/C5} {
			\draw[thick] (\Ci) ellipse [x radius=\R, y radius=\r];
			\node at (\Ci) {$K_{n_{\i}}$};
		}
		
		\foreach \i/\A/\B in {1/C1/C2,2/C2/C3,3/C3/C4,4/C4/C5} {
			\path let \p1=(\A), \p2=(\B) in coordinate (S\i) at ($(\p1)!.5!(\p2)$);
			\fill (S\i) circle (2.2pt);
			\node[below=7pt] at (S\i) {$v_{\i}$};
		}
		
		\foreach \i/\A/\B in {1/C1/C2,2/C2/C3,3/C3/C4,4/C4/C5} {
			\draw (S\i) -- ++(0,0); 
		}
		
		
	\end{tikzpicture}
	\caption{The general Wollastonite graph $W(n_1,n_2,\dots,n_5).$}
	\label{wollostonite}
\end{figure}

\begin{lemma}
 $K[\Delta(n_1,\ldots,n_e)]=K[x_1,\ldots,x_N]/I$, $N=\sum_{i=1}^en_i-(e-1)$, where $I$ is generated by all $x_ix_j$ for which $x_i$ and $x_j$ 
 do not belong to the same simplex.
 \end{lemma}
 
 \begin{proof}
We know that $K[\Delta(n_1,\ldots,n_e)]$ has a 2-linear resolution, so in particular, if $K[\Delta(n_1,\ldots,n_e)]=K[x_1,\ldots,x_N]/I$, 
$N=\sum_{i=1}^en_i-(e-1)$, then $I$ is generated in degree 2. It is clear that $x_ix_j\in I$ if and only if $x_i$ and $x_j$ do not belong to the same simplex.
\end{proof}

\begin{theorem}
If $1\le k<\dim(\Delta(n_1,\ldots,n_e))=\max\{n_i\}-1$, then $K[\Delta(n_1,\ldots,n_e)^{(k)}]=K[x_1,\ldots,x_N]/I$, where $I$ is generated in degree 2 and $k+2$.
If $k\ge\dim(\Delta(n_1,\ldots,n_e))=\max\{n_i\}-1$, then $I$ is generated in degree 2. The generators of degree 2 are the same for all $k\ge1$.
\end{theorem}

\begin{proof}
The generators of $I$ which are of degree 2 are determined by the 1-skeleton, which is the same for all $k\ge1$. 
If $k\ge\dim\Delta(n_1,\ldots,n_e)$, then $\Delta(n_1,\ldots,n_e)^{(k)}=
\Delta(n_1,\ldots,n_e)$, so, for each $k$, the generators of degree 2 are all $x_ix_j$ for which $x_i$ and $x_j$ do not belong to the same simplex.
The remaining generators of $I$ are of the form $x_{i_1}\cdots x_{i_l}$, where all $x_{i_j}$, $1\le i\le l$, belong to the same simplex. Such an element is a
generator for $I$ when $\{x_{i_1},\ldots,x_{i_l}\}$ does not belong to $\Delta(n_1,\ldots,n_e)^{(k)}$, but all subset of $\{x_{i_1},\ldots,x_{i_l}\}$ with $l-1$ elements do.
The only possibility is that $\{x_{i_1},\ldots,x_{i_l}\}$ is a subsimplex $\tau$ of dimension $k+1$, i.e., with $k+2$ elements, of a simplex $\sigma$ of dimension 
$\ge k+1$ in $\Delta(n_1,\ldots,n_e)$. If $k\ge\max\{n_i\}-1$, then $K[\Delta(n_1,\ldots,n_e)^{(k)}]=K[\Delta(n_1,\ldots,n_e)]$ which has a 2-linear resolution,
in particular has relations of degree 2.
\end{proof}

For $k<\max\{n_i\}-1$ write $I=I_2+I_{k+2}$, where $I_2$ is generated by the minimal generators for $I$ of degree 2, and $I_{k+2}$ is generated by the minimal generators for $I$ of degree $k+2$.

\begin{theorem}
We have that $K[x_1,\ldots,x_N]/I_2$ has a linear resolution.
\end{theorem}

\begin{proof}
$K[x_1,\ldots,x_N]/I_2=K[\Delta(n_1,\ldots,n_e)]$ which has a 2-linear resolution since $\Delta(n_1,\ldots,n_e)$ is a fat forest.
\end{proof}

We are now ready to determine all Betti numbers of $K[\Delta(n_1,\ldots,n_e)^{(k)}]$. We will use that for a graded ring $R=K[x_1,\ldots,x_N]/I$, the Hilbert series
of $R$ equals $$R(t)=\sum_{i=0}^N(-1)^i\beta_{i,j}(R)t^j)/(1-t)^N.$$

\begin{theorem}
If $k\ge1$ we have that $\beta_{i,j}(K[\Delta(n_1,\ldots,n_e)^{(k)}])=\beta_{i,j}(K[\Delta(n_1,\ldots,n_e)])$ if $j-i<k+1$, and $\beta_{i,j}(K[\Delta(n_1,\ldots,n_e)^{(k)}])=0$
if $j-i>k+1$.
\end{theorem}

\begin{proof}
$\Delta(n_1,\ldots,n_e)^{(k)}$ is identical to $\Delta(n_1,\ldots,n_e)$ in dimension $\le k$, so for all induced subcomplexes $(\Delta(n_1,\ldots,n_e))_S$ 
with $|S|\le k+1$ we have
$\widetilde H_s((\Delta(n_1,\ldots,n_e)^{(k)})_S)=\widetilde H_s((\Delta(n_1,\ldots,n_e))_S)$
if $s<k$. Thus, according to Hochster's theorem, 
$$\beta_{i,j}(K[\Delta(n_1,\ldots,n_e)^{(k)}])=
\sum_{\overset{S\subseteq V}{|S|=j}}\dim_K\widetilde H_{j-i-1}(\Delta(n_1,\ldots,n_e)_S;K)=$$
$$\sum_{\overset{S\subseteq V}{|S|=j}}
\dim_K\widetilde H_{j-i-1}(\Delta(n_1,\ldots,n_e)^{(k)})_S;K)=\sum\beta_{i,j}(K[\Delta(n_1,\ldots,n_e)])$$ 
if $j-i<k+1.$
Furthermore $\Delta(n_1,\ldots,n_e)^{(k)}=0$ in dimension $>k$, so we get
$\beta_{i,j}(K[\Delta(n_1,\ldots,n_e)^{(k)}])=0$ if $j-i>k+1$, using Hochster's theorem in the same way.
\end{proof}

\begin{theorem}\label{main}
The nonzero Betti numbers of $K[\Delta(n_1,\ldots,n_e)^{(k)}]$ are $\beta_{0,0}=1$, and for $i\ge1$ 
$$\beta_{i,i+1}(K[\Delta(n_1,\ldots,n_e)^{(k)}=-\sum_{s=1}^e{N-n_s\choose i+1}+(e-1){N-1\choose i+1}$$
where
$$N=\sum_{s=1}^en_s-(e-1),$$
and, if $k<n-1$, 
$$\beta_{i,k+1+i}(K[\Delta(n_1,\ldots ,n_e)^{(k)}])=\sum_{j=k+2}^n\sum_{s=1}^e{n_s\choose j}(-1)^{k-j}{N-j\choose k+i+1-j},$$
where $n=\max\{n_i\}$.
\end{theorem}

\begin{proof}
For $j-i<k+1$ we can calculate $\beta_{i,j}(K[\Delta(n_1,\ldots,n_e)^{(k)}])$ from the Hilbert series of $K[\Delta(n_1,\ldots,n_e)])=K[x_1,\ldots,x_N]/I_2$.
We know that $K[x_1,\ldots,x_N]/I_2$ has a linear resolution, so the only nonzero Betti numbers if $i>0$ are $\beta_{i,i+1}$.
Let $N=\sum_{s=1}^en_s-(e-1)$ and $n=\max\{n_i\}$. The Hilbert series of $K[x_1,\ldots,x_N]/I_2$ is according to \cite[Theorem 1]{fr1}
$$\sum_{s=1}^e\frac{1}{(1-t)^{n_s}}-\frac{(e-1)}{1-t}=\frac{\sum_{s=1}^e(1-t)^{N-n_s}-(e-1)(1-t)^{N-1}}{(1-t)^N}=\frac{\sum_{i=0}^Nc_it^i}{(1-t)^N}.$$
Now $\beta_{i,i+1}$ is $(-1)^ic_{i+1}$, and 
$$c_{i+1}=(-1)^{i+1}(\sum_{s=1}^e{N-n_s\choose i+1}-(e-1){N-1\choose i+1}).$$
The Betti numbers for $j-i=k+1$ are, according to Lemmas \ref{fvector} and \ref{size}, determined from the Hilbert series of $k[\Delta(n,e)^{(k)}]$, which is
$$1+\frac{Nt}{1-t}+\sum_{j=2}^{k+1}\frac{\sum_{s=1}^e{n_s\choose j}t^j}{(1-t)^j}=
\frac{(1-t)^N+Nt(1-t)^{N-1}+\sum_{j=2}^{k+1}\sum_{s=1}^e{n_s\choose j}t^j(1-t)^{N-j}}{(1-t)^N}.$$
To get the Betti numbers $\beta_{i,i+1+k}$ we subtract the part that gives the Betti numbers $\beta_{i,i+1}$, 
thus the part that comes from the Hilbert series of  $K[\Delta(n,e)]$. We now write this series as
$$\frac{(1-t)^N+Nt(1-t)^{N-1}+\sum_{j=2}^n\sum_{s=1}^e{n_s\choose j}t^j(1-t)^{N-j}}{(1-t)^N},$$
so we can read off the Betti numbers $\beta_{i,i+k+1}(K[\Delta(n,e)_{(k)}])$ from
$$\frac{-\sum_{j=k+2}^nt^j\sum_{s=1}^e{n_s\choose j}(1-t)^{N-j}}{(1-t)^N}=\sum_{s=k+2}^Nd_st^s.$$
Now $\beta_{i,i+k+1}=(-1)^id_{i+k+1}.$
We get 
$$d_{k+i+1}=-\sum_{j=k+2}^n\sum_{s=1}^e{n_s\choose j}(-1)^{k+i+1-j}{N-j\choose k+i+1-j}.$$
\end{proof}

We specialize to the case when $n_i=n$ for all $i$.

\begin{corollary}
The nonzero Betti numbers of $K[\Delta(n,e)^{(k)}]$ are $\beta_{0,0}=1$, and for $i\ge1$ 
$$\beta_{i,i+1}=-e{N-n\choose i+1}+(e-1){N-1\choose i+1}$$
and, if $k<n-1$, $\beta_{i,k+1+i}(K[\Delta(n,e)^{(k)}])=\sum_{j=k+2}^ne{n\choose j}(-1)^{k-j}{N-j\choose k+i+1-j}$.
\end{corollary}

\begin{corollary}
We have pd($K[\Delta(n,e)^{(k)})]=N-2$ and reg($K[\Delta(n,e)^{(k)}])=k+1$ if $1\le k<\dim(\Delta(n,e))$. If $k\ge\dim(\Delta(n,e))$, then reg($K[\Delta(n,e)^{(k)}])=1$.
\end{corollary}

\begin{proof}
If $k\ge\dim(\Delta(n,e))$, then $K[\Delta(n,e)^{(k)}]=K[\Delta(n,e)]$, so reg($K[\Delta(n,e)^{(k)}])=1$ since $\Delta(n,e)$ is a fat forest. The remaining claims follow
directly from Theorem \ref{main}.
\end{proof}

\begin{corollary}
If $k\le1$, then $K[\Delta(n,e)^{(k)}]$ is Cohen-Macaulay for all sets of $n_i$. If $k>1$, then $K[\Delta(n,e)^{(k)}]$ is Cohen-Macaulay only if 
$n=2$.
\end{corollary}

\begin{proof}
If $k\le1$ we use Reisner's criterion \cite{re} saying that $k[\Sigma]$ is Cohen-Macaulay if and only if for each face $\sigma\in\Sigma$, $\widetilde H_i({\rm link}_\sigma)=0$ 
for all $i<\dim({\rm link}_\sigma)$. If $k>1$ and $n=2$, then $K[\Delta(n,e)^{(k)}]=K[\Delta(n,e)]$ which is Cohen-Macaulay according to Reisner's criterion. If $k>1$ and
$n>2$, then depth($K[\Delta(n,e)^{(k)}])=2$ and $\dim(K[\Delta(n,e)^{(k)}])\ge3$. 
\end{proof}

We illustrate the results with an example. Consider $K[\Delta(3,4,5)^{(k)}]$. Here is the result for $k=1$. Row 1 contains $\beta_{i,i+1}$, row 2 contains $\beta_{i,i+2}$.

 \medskip     
$\begin{matrix}
      &0&1&2&3&4&5&6&7&8\\
      \text{total:}&1&40&195&456&629&540&285&85&11\\
      \text{0:}&1&\text{.}&\text{.}&\text{.}&\text{.}&\text{.}&\text{.}&\text{.}&\text{.}\\
      \text{1:}&\text{.}&26&103&197&224&160&71&18&2\\
      \text{2:}&\text{.}&15&99&280&440&415&235&74&10\\
      \end{matrix}$

\medskip
We show how to determine the Betti numbers from Theorem \ref{main}. We have $\beta_{i,i+1}(K[\Delta(3,4,5)^{(1)}]=\beta_{i,i+1}(K[\Delta(3,4,5)]$,
and $K[\Delta(3,4,5)]$ has Hilbert series 
$$1+\frac{10t}{1-t}+\frac{19t^2}{(1-t)^2}+\frac{15t^3}{(1-t)^3}+\frac {6t^4}{(1-t)^4}+\frac{t^5}{(1-t)^5}=$$
$$\frac{1-26t^2+103t^3-197t^4+224t^5-160t^6+71t^7-18t^8+2t^9}{(1-t)^{10}}=\frac{1+\sum_{i=1}^9c_it^i.}{(1-t)^{10}}$$
Now $\beta_{i,i+1}=(-1)^ic_{i+1}$. To get the remaining Betti numbers we determine the Hilbert series of $K[\Delta(3,4,5)^{(1)}]$, which is
$$1+\frac{10t}{1-t}+\frac{19t^2}{(1-t)^2}=\frac{1-26t^2+88t^3-98t^4-56t^5+280t^6-344t^7+217t^8-72t^9+10t^{10}}{(1-t)^{10}}=
\frac{\sum_{i,j}(-1)^i\beta_{i,j}t^j}{(1-t)^{10}}.$$
We have $\beta_{2,3}-\beta_{1,3}=88$. Since $\beta_{2,3}=103$, we get $\beta_{2,3}=15$. We have $\beta_{2,4}-\beta_{3,4}=-98$. Since
$\beta_{3,4}=197$, we get $\beta_{2,4}=99$ a.s.o

\medskip\noindent
Here is the result for $K[\Delta(3,4,5)_{(2)}]$.

\medskip
$\begin{matrix}
      &0&1&2&3&4&5&6&7&8\\
      \text{total:}&1&32&138&282&334&240&102&23&2\\
      \text{0:}&1&\text{.}&\text{.}&\text{.}&\text{.}&\text{.}&\text{.}&\text{.}&\text{.}\\
      \text{1:}&\text{.}&26&103&197&224&160&71&18&2\\
      \text{2:}&\text{.}&\text{.}&\text{.}&\text{.}&\text{.}&\text{.}&\text{.}&\text{.}&\text{.}\\
      \text{3:}&\text{.}&6&35&85&110&80&31&5&\text{.}\\
      \end{matrix}$

\medskip
And here for $K[\Delta(3,4,5)_{(3)}]$.

\medskip
$\begin{matrix}
      &0&1&2&3&4&5&6&7&8\\
      \text{total:}&1&27&108&207&234&165&72&18&2\\
      \text{0:}&1&\text{.}&\text{.}&\text{.}&\text{.}&\text{.}&\text{.}&\text{.}&\text{.}\\
      \text{1:}&\text{.}&26&103&197&224&160&71&18&2\\
      \text{2:}&\text{.}&\text{.}&\text{.}&\text{.}&\text{.}&\text{.}&\text{.}&\text{.}&\text{.}\\
      \text{3:}&\text{.}&\text{.}&\text{.}&\text{.}&\text{.}&\text{.}&\text{.}&\text{.}&\text{.}\\
      \text{4:}&\text{.}&1&5&10&10&5&1&\text{.}&\text{.}\\
      \end{matrix}$

\section{Competing interests}
There are no competing interests.

\end{document}